\documentclass[conference]{IEEEtran}
\IEEEoverridecommandlockouts
\usepackage{cite}

\usepackage{mathtools}
\usepackage{amsmath,amssymb,amsfonts,amsthm}
\usepackage{bm}
\newtheorem{thm}{Theorem}[section]
\newtheorem{prop}[thm]{Proposition}
\newtheorem{lem}[thm]{Lemma}

\theoremstyle{definition}
\newtheorem{definition}[thm]{Definition}

\newcommand{\lsqparen}{%
	\bm{[}\kern-2pt{\scalebox{.5}[1]{\bm{$($}}}}
	
\newcommand{\rsqparen}{%
	\raisebox{-.4pt}[.1pt]{\scalebox{.5}[1]{\bm{$)$}}}\kern-2pt\bm{]}}

\newcommand{\biglsqparen}{%
	\bigl[\kern-4pt\bigl(}
	
\newcommand{\Biglsqparen}{%
	\Bigl[\kern-4.3pt\Big(}
	
\newcommand{\bigglsqparen}{%
	\biggl[\kern-4.9pt\bigg(}
	
\newcommand{\Bigglsqparen}{%
	\bm{\Biggl[}\kern-5pt\raisebox{0pt}[2pt]{\scalebox{.7}[1]{\bm{$\Bigg($}}}}
	
\usepackage{hyperref}

\begin{document}

\title{Linear Systems over Join-Blank Algebras}

\author{\IEEEauthorblockN{ Hayden Jananthan}
\IEEEauthorblockA{\textit{Mathematics Department} \\
\textit{Vanderbilt University}\\
Nashville, Tennessee \\
hayden.r.jananthan@vanderbilt.edu}
\and
\IEEEauthorblockN{Suna Kim}
\IEEEauthorblockA{\textit{Mathematics Department} \\
\textit{California Institute of Technology}\\
Pasadena, California \\
skim3@caltech.edu}
\and
\IEEEauthorblockN{Jeremy Kepner}
\IEEEauthorblockA{\textit{Lincoln Laboratory Supercomputing Center} \\
\textit{Massachusetts Institute of Technology}\\
Lexington, Massachusetts \\
kepner@ll.mit.edu}

}

\maketitle

\begin{abstract}
A central problem of linear algebra is solving linear systems. Regarding linear systems as equations over general semirings $(V,\oplus,\otimes,0,1)$ instead of rings or fields makes traditional approaches impossible. Earlier work shows that the solution space $X(\mathbf{A},\mathbf{w})$ of the linear system $\mathbf{A}\mathbf{v}=\mathbf{w}$ over the class of semirings called join-blank algebras is a union of closed intervals (in the product order) with a common terminal point. 

In the smaller class of max-blank algebras, the additional hypothesis that the solution spaces of the $1\times 1$ systems $A\otimes v = w$ are closed intervals implies that $X(\mathbf{A},\mathbf{w})$ is a finite union of closed intervals. We examine the general case, proving that without this additional hypothesis, we can still make $X(\mathbf{A},\mathbf{w})$ into a finite union of \emph{quasi-intervals}.

\end{abstract}


\section{Introduction}

Linear algebra is a cornerstone of modern computation. In particular, one approach to solving a problem in application is reducing it to a linear-algebraic problem, such as carrying out a matrix multiplication, solving a linear system, or finding eigenvalues and eigenvectors.

As data become more varied, new mathematics must be developed to handle linear algebra over more general algebraic structures. For example, the need for a variety of data types to be supported exists in the context of polystore databases \cite{kepner2016} and prompted the creation of the Dynamic Distributed Dimensional Data Model (D4M) \cite{kepner2012} which provides a linear algebraic interface to graphs stored in NoSQL \cite{byun2012,kepner2013}, SQL \cite{wu2014,gadepally2015}, and NewSQL \cite{samsi2016}.

One of the most general algebraic structures over which linear algebra makes sense is a semiring. Semirings include many algebraic structures that we often encounter -- in particular, all rings and fields are semirings. Among the most studied semirings which are not rings include the \emph{max-plus algebra} $\mathbb{R}\cup \{\infty,-\infty\}$, which forms a semiring with addition $\max$ and multiplication $+$, and the \emph{max-min algebra} $\mathbb{R}\cup \{\infty,-\infty\}$, which forms a semiring with addition $\max$ and multiplication $\min$ \cite{akian2006, litvinov2009}.

In fact, mathematicians and scientist have found numerous applications of max-plus algebra; it is widely used fields like in performance evaluation of manufacturing systems, discrete event system theory, Markov decision processes, and even in language theory \cite{gaubert1997}.

Note that a semiring generalizes the notion of a ring by dropping the necessity of additive inverses existing. One method of dealing with the loss of subtraction is to make use of non-algebraic properties, particularly strong order-theoretic properties \cite{han2012}. Thus we focus on the semirings that are induced from ordered sets (as in max-plus algebra) and utilize those properties to characterize the solution set.

\section{Definitions}

The most basic object of study is that of a semiring. 

\begin{definition}[Semiring]
\cite{golan1999, gondran2007} A \emph{semiring} is a quintuple $(V,\oplus,\otimes,0,1)$ consisting of
\begin{enumerate}
\item an underlying set $V$,
\item two binary operations $\oplus$ (\emph{addition}) and $\otimes$ (\emph{multiplication}) on $V$, and
\item two elements $0$ and $1$ of $V$
\end{enumerate}
such that
\begin{enumerate}
\item $\oplus$ is associative, commutative, and has identity element $0$,
\item $\otimes$ is associative and has identity element $1$,
\item $\otimes$ distributes over $\oplus$, and
\item $0$ is a multiplicative annihilator.
\end{enumerate}
\end{definition}

Matrices and their operations can be defined over general semirings, in a similar way it is over fields like $\mathbb{R}$ or $\mathbb{C}$.

\begin{definition}[Matrices]
An $m\times n$ \emph{matrix} over a semiring $V$ is a map 
\begin{equation*}
\mathbf{A}: \{1,\ldots,m\} \times \{1,\ldots, n\} \to V
\end{equation*}

If $\mathbf{A}$ and $\mathbf{B}$ are two $m\times n$ matrices, their \emph{sum} is the $m\times n$ matrix $\mathbf{A}\oplus \mathbf{B}$ defined by
\begin{equation*}
(\mathbf{A}\oplus \mathbf{B})(i,j) = \mathbf{A}(i,j) \oplus \mathbf{B}(i,j)
\end{equation*}

If $\mathbf{A}$ is an $m\times n$ matrix and $\mathbf{B}$ is an $n\times p$ matrix, their \emph{product} is the $m\times p$ matrix $\mathbf{A}\mathbf{B}$ defined by
\begin{equation*}
\mathbf{A}\mathbf{B}(i,j) = \bigoplus_{k=1}^n{\mathbf{A}(i,k)\otimes \mathbf{B}(k,j)}
\end{equation*}
\end{definition}

Elements of the Cartesian product $V^n$ are identified with $n\times 1$ matrices over $V$.

\begin{definition}[Linear Systems]
An $m\times n$ \emph{linear system} over $V$ is an equation of the form $\mathbf{A}\mathbf{v} =\mathbf{w}$ where  $\mathbf{A}$ is a fixed $m\times n$ matrix, $\mathbf{w}$ is a fixed $m\times 1$ matrix, and $\mathbf{v}$ is a variable $n\times 1$ matrix.

The \emph{solution space} $X(\mathbf{A},\mathbf{w})$ of a linear system $\mathbf{A}\mathbf{v}=\mathbf{w}$ is the set
\begin{equation*}
X(\mathbf{A},\mathbf{w}) = \{ \mathbf{v} \mid \mathbf{A}\mathbf{v} =\mathbf{w}\}
\end{equation*}
\end{definition}

One nice class of semirings which is diametrically opposite of the notion of a ring is that of join-blank algebras, which make explicit and extended use of an underlying order by requiring that the underlying set be a complete lattice, the addition operation be binary supremum, and the multiplication operation satisfy an ``infinite-distributivity'' law.

\begin{definition}[Complete Lattice]
A pair $(V,\leq)$ of a set $V$ and a binary relation $\leq$ on $V$ is a \emph{complete lattice} if
\begin{enumerate}
\item $\leq$ is reflexive, antistymmetric, and transitive,
\item for any subset $U\subset V$ there exists a least element $\bigvee U$ greater than or equal to every element of $U$, called the \emph{join} or \emph{supremum} of $U$, and
\item for any subset $U\subset V$ there exists a greatest element $\bigwedge U$ less than or equal to every element of $U$, called the \emph{meet} or \emph{infimum} of $U$.
\end{enumerate}
\end{definition}

In the case of a two element set $\{u,v\}$, the join of $\{u,v\}$ is denoted
\begin{equation*}
\bigvee\{u,v\} = u \vee v 
\end{equation*}
and its meet is denoted
\begin{equation*}
\bigwedge\{u,v\} = u \wedge v
\end{equation*}
These binary operations $\vee$ and $\wedge$ are called \emph{join} and \emph{meet}, respectively. Semirings in which the underlying set and the operations have order-theoretic properties with respect to a fixed partial order allows for order-theoretic tools to be applied to the construction of solution sets in terms of intervals.

\begin{definition}[Join-Blank Algebra]
A \emph{join-blank algebra} is a semiring $(V,\vee,\otimes,-\infty,1)$ where
\begin{enumerate}
\item $V$ is a complete lattice with respect to some fixed order,
\item $\vee$ is the join with respect to that order,
\item $-\infty$ is the minimum element of $V$ with respect to that order, and
\item for any subset $U\subset V$ and element $v\in V$
\begin{equation*}
v \wedge \bigvee{U} = \bigvee \{ v \wedge u \mid u \in U\}
\end{equation*}
\end{enumerate}
\end{definition}

The max-plus algebra $(\mathbb{R}\cup \{-\infty,\infty\},\max,+,-\infty,0)$ and the max-min algebra $(\mathbb{R}\cup \{-\infty,\infty\},\max,\min,-\infty,\infty)$ are join-blank algebras.  Power set algebras $(\mathcal{P}(S),\cup,\cap,\emptyset,S)$ and more generally Heyting algebras form join-blank algebras. \cite{kepnerjananthan}

The order-theoretic properties of a join-blank algebra $V$ can be extended to the Cartesian product $V^n$.

\begin{definition}[Product Order]
Suppose $V$ is ordered by $\leq$. The \emph{product order} $\leq$ on $V^n$ is defined by
\begin{equation*}
\mathbf{v} \leq \mathbf{w} \quad \text{if and only if} \quad \mathbf{v}(i) \leq \mathbf{w}(i) \quad \text{for all $i$}
\end{equation*}
\end{definition}

\section{Join-Blank Structure Theorem}

The order-theoretic properties of a join-blank algebra $V$ extend to order-theoretic properties of $V^n$.

\begin{prop} \label{order-theoretic properties of product order} \cite{kepnerjananthan}
If $V$ is a complete lattice, then $V^n$ is a complete lattice. Moreover, if $U\subset V^n$ then $\bigvee{U}$ exists if and only if $\bigvee\{\mathbf{v}(i) \mid \mathbf{v} \in U \}$ exists for each $i$, in which case
\begin{equation*}
\left( \bigvee{U} \right)(i) = \bigvee\{\mathbf{v}(i) \mid \mathbf{v} \in U\}
\end{equation*}
\end{prop}

The compatibility of $\vee$ and $\otimes$ with the order contribute to order-theoretic properties of $X(\mathbf{A},\mathbf{w})$:

\begin{prop} \label{order-theoretic properties of solution space} \cite{kepnerjananthan}
Suppose $\mathbf{A}\mathbf{v}=\mathbf{w}$ is a linear system over a join-blank algebra $V$. 
\begin{enumerate}[(a)]
\item $X(\mathbf{A},\mathbf{w})$ is closed under taking joins of non-empty subsets.
\item $X(\mathbf{A},\mathbf{w})$ is convex, so if $\mathbf{v}_1 \leq \mathbf{v}_2 \leq \mathbf{v}_3$ and $\mathbf{v}_1,\mathbf{v}_3 \in X(\mathbf{A},\mathbf{w})$, then $\mathbf{v}_2 \in X(\mathbf{A},\mathbf{w})$.
\end{enumerate}
\end{prop}

This implies the following structure of $X(\mathbf{A},\mathbf{w})$ as a union of closed intervals with a common terminal point. Recall that a \emph{closed interval} is defined as
\begin{equation*}
[\mathbf{x},\mathbf{y}] = \{ \mathbf{z} \mid \mathbf{x} \leq \mathbf{z} \leq \mathbf{y}\}
\end{equation*}

\begin{thm}[Join-Blank Structure Theorem] \label{join-blank structure theorem} \cite{kepnerjananthan}
Suppose $\mathbf{A}\mathbf{v}=\mathbf{w}$ is a linear system over a join-blank algebra $V$. Then $X(\mathbf{A},\mathbf{w})$ is of the form 
\begin{equation*}
X(\mathbf{A},\mathbf{w}) = \bigcup_{\mathbf{v} \in U}{[\mathbf{v},\mathbf{x}]}
\end{equation*}
for some $U \subset X(\mathbf{A},\mathbf{w})$ and a fixed $\mathbf{x}$.
\end{thm}

This structure allows for the problem of finding the solution space $X(\mathbf{A},\mathbf{w})$ to be reduced to finding the solution spaces 
\begin{equation*}
X(\mathbf{A}(i,:),\mathbf{w}(i))
\end{equation*}
to the single-equation linear systems 
\begin{equation*}
\mathbf{A}(i,:) \mathbf{v} = \mathbf{w}(i)
\end{equation*}
Intersecting each solution set will give us the complete solution set for the original linear system, as the intersection will satisfy all equations.

\section{Max-Blank Structure Theorem}

When $V$ is \emph{totally-ordered} and hence $\vee = \max$, we call $V$ a \emph{max-blank algebra}. 

In nice cases, the solution space of a linear system over a max-blank algebra is a finite union of closed intervals.  However, it is not always the case. Consider the following system in max-blank algebra:

\begin{equation*}
\left[ \begin{matrix} \infty \end{matrix} \right] \left[ \begin{matrix} v \end{matrix}\right] = \left[ \begin{matrix} \infty \end{matrix}\right]
\end{equation*}
The solution space is given by
$(-\infty,\infty]$, which cannot be  written as finite unions of intervals.
In the $1\times 1$ case, the solution set can be represented with finitely many closed intervals if
\begin{equation*}
X(A,w) = \{ v \mid A \otimes v = w\}
\end{equation*}
is a closed interval.

\begin{thm}[Max-Blank Structure Theorem for Closed Intervals] \label{max-blank structure theorem for closed intervals} \cite{kepnerjananthan}
Suppose $\mathbf{A}\mathbf{v} = \mathbf{w}$ is a linear system over a max-blank algebra such that for every $i,j$ the set
\begin{equation*}
\{ v \mid \mathbf{A}(i,j) \otimes \mathbf{w}(i)\}
\end{equation*}
is a closed interval. Then $X(\mathbf{A},\mathbf{w})$ is a finite union of closed intervals.
\end{thm}

The crucial step in the proof of Theorem~\ref{max-blank structure theorem for closed intervals} is that a Cartesian product of closed intervals in $V$ is a closed interval in $V^n$ in the product order.

Proposition~\ref{order-theoretic properties of solution space} shows that the only other form that
\begin{equation*}
\{ v \mid \mathbf{A}(i,j) \otimes \mathbf{w}(i)\}
\end{equation*}
can take on is a half-open interval which is open on the left. The Cartesian product of arbitrary intervals in $V$ need to be an interval in $V^n$ in the product order.

This motivates a slightly more general basic object than intervals.

\begin{definition}[Quasi-interval]
Suppose $I_1,\ldots, I_n$ are intervals in $V$. Then define
\begin{equation*}
\prescript{}{A}{\lsqparen} \mathbf{p},\mathbf{q} {\rsqparen}_{\raisebox{-1pt}{\footnotesize $B$}} = I_1 \times \cdots \times I_n
\end{equation*}
where $I_k$ has endpoints $\mathbf{p}(k)$ and $\mathbf{q}(k)$, with exclusion of $\mathbf{p}(k)$ when $k\in A$ and exclusion of $\mathbf{q}(k)$ when $k\in B$.
\end{definition}

\begin{lem} \label{intersection of quasi-intervals}
Suppose $V$ is totally ordered. Suppose $A,B \subset \{1,\ldots,n\}$ and $\mathbf{p},\mathbf{q},\mathbf{r},\mathbf{s} \in V^n$. Then
\[
\prescript{}{A}{\lsqparen} \mathbf{p},\mathbf{q} ] \cap \prescript{}{B}{\lsqparen} \mathbf{r},\mathbf{s}] = \prescript{}{C}{\lsqparen} \mathbf{p} \vee \mathbf{r}, \mathbf{q} \wedge \mathbf{s} ]
\]
where
\begin{align*}
C & = \{ i \in A \setminus B \mid \mathbf{p}(i) \geq \mathbf{r}(i)\} \\
& \quad \cup \{j \in B \setminus A \mid \mathbf{p}(j) \leq \mathbf{r}(j)\} \\
& \quad \cup A\cap B
\end{align*}
\end{lem}

\begin{proof}
Let 
\[
\prescript{}{A}{\lsqparen} \mathbf{p},\mathbf{q}] = I_1 \times \cdots \times I_n
\]
and
\[
\prescript{}{B}{\lsqparen} \mathbf{r},\mathbf{s}] = J_1 \times \cdots \times J_n
\]
Then
\begin{align*}
\prescript{}{A}{\lsqparen} \mathbf{p}, \mathbf{q}] \cap \prescript{}{B}{\lsqparen} \mathbf{r},\mathbf{s}] &= (I_1 \times \cdots \times I_n) \cap (J_1 \times \cdots \times J_n)
\\
&= (I_1 \cap J_1) \times \cdots \times (I_n\cap J_n)
\end{align*}
Since the intersection of intervals is also interval, each of $I_k \cap J_k$ is an interval with end-points $\mathbf{p}(k) \vee \mathbf{r}(k)$ and $\mathbf{q}(k) \wedge \mathbf{s}(k)$ with exclusion of the first end-point $\mathbf{p}(k) \vee \mathbf{r}(k)$ exactly when either
\begin{enumerate}[(i)]
\item $k \in A \setminus B$ and $\mathbf{p}(k) \geq \mathbf{r}(k)$, or
\item $k\in B \setminus A$ and $\mathbf{p}(k) \leq \mathbf{r}(k)$, or
\item $k \in A \cap B$.
\end{enumerate}
and inclusion of the second end-point $\mathbf{q}(k) \wedge \mathbf{s}(k)$. Hence
\[
\prescript{}{A}{\lsqparen} \mathbf{p}, \mathbf{q}] \cap \prescript{}{B}{\lsqparen} \mathbf{r},\mathbf{s}] = \prescript{}{C}{\lsqparen} \mathbf{p}, \mathbf{q}] 
\]
where
\begin{align*}
C & = \{ i \in A \setminus B \mid \mathbf{p}(i) \geq \mathbf{r}(i)\} \\
& \quad \cup \{j \in B \setminus A \mid \mathbf{p}(j) \leq \mathbf{r}(j)\} \\
& \quad \cup (A\cap B)
\end{align*}
\end{proof}

Using this new notation, we show that linear systems in max-blank algebra have solution set that can be written as union of finite quasi-interval.

\begin{thm}[Max-blank Structure Theorem] \label{max-blank structure theorem}
Suppose $\mathbf{A}$ is an $n\times  m$ matrix and $\mathbf{w}$ an element of $V^n$. Let $U_i$, $U_{i,1}$, and $U_{i,2}$ be the sets of $j\in \{1,\ldots, m\}$ such that
\[
  X(\mathbf{A}(i,j),\mathbf{w}(i)) = \{ v \in V \mid \mathbf{A}(i,j)\otimes v = \mathbf{w}(i)\}
\]
is non-empty, non-empty and not a closed interval, and non-empty and a closed interval, respectively.  When $j\in U_{i,1}$, let
  $X(\mathbf{A}(i,j),\mathbf{w}(i))=(p_j^i,q_j^i]$. For $j\in U_{i,2}$, let
$X(\mathbf{A}(i,j),\mathbf{w}(i)) = [p_j^i,q_j^i]$.
Lastly, for $j\notin U_i$ let $q_j^i$ be the largest element such that
  $\mathbf{A}(i,j)\otimes q_j^i\leq \mathbf{w}(i)$.
  \\
Let $\mathbf{p}_{i,j'}$ be defined by
\[
  \mathbf{p}_{i,j'}(j)=\begin{cases} p_{j'}^i & j=j' \\ - \infty & \text{otherwise} \end{cases}
\]
and $\mathbf{q}_i$ be defined by
  $\mathbf{q}_i(j)=q_j^i$. Then
\[
X(\mathbf{A},\mathbf{w})=\bigcup_{\substack{j'_{i,k} \in U_{i,k} \\ \text{for $1\leq i\leq n$} \\ \text{and $k\in \{1,2\}$}}}{\prescript{}{J_\mathbf{j}}{\Bigglsqparen} \bigvee_{1\leq i\leq n, k\in \{1,2\}}{\mathbf{p}_{i,j_{i,k}'}}, \bigwedge_{1\leq i\leq n}{\mathbf{q}_i} \Biggr]}
\]
where $\mathbf{j} = (j'_{i,k})_{1\leq i\leq n, k\in \{1,2\}}$ and $\ell \in J_\mathbf{j}$ if and only if 
\[
\max_{1\leq i\leq n}{\mathbf{p}_{i,j_{i,2}}} \leq \max_{1\leq i\leq n}{\mathbf{p}_{i,j_{i,1}}}
\]
\end{thm}

\begin{proof}
The proof will consist of finding the solution space of 
$\max_{j\in \{1,\ldots, m\}}{(\mathbf{A}(i,j)\otimes \mathbf{v}(j))}=\mathbf{w}(i)$
for each $i\in \{1,\ldots, n\}$ and showing that it is the union of intervals with a common (inclusive) terminal point. Taking the intersection of these solution spaces is $X(\mathbf{A},\mathbf{w})$.

$X([v],[u])$ is convex with a inclusive terminal point. Since $V$ is a complete lattice, it follows that 
\[
X(\mathbf{A}(i,j),\mathbf{w}(i)) = [p_j^i,q_j^i]) \quad \text{or} \quad X(\mathbf{A}(i,j),\mathbf{w}(i)) = (p_j^i,q_j^i] 
\]
for some $p_j^i$ and $q_j^i$, assuming that $X(\mathbf{A}(i,j),\mathbf{w}(i))$ is non-empty. 

Let define $U_i$, $U_{i,1}$, and $U_{i,2}$ as in the proposition. For $j\notin U_i$, let $q_j^i$ be the largest element such that $\mathbf{A}(i,j) \otimes q_j^i \leq \mathbf{w}(i)$. Such an element exists since $-\infty \otimes v = -\infty$ and multiplication by a fixed element is a monotonic map.

The solution space now can be written down in terms of the elements $p_j^i$ and $q_j^i$. A given $\mathbf{v}$ is in the solution set if and only if there exists a $j'\in \{1,\ldots, m\}$ such that
\[
  \mathbf{A}(i,j')\otimes \mathbf{v}(j')=\mathbf{w}(i)
\]
and for all $j\in \{1,\ldots, m\}$ it is true that
\[
  \mathbf{A}(i,j)\otimes \mathbf{v}(j)\leq \mathbf{w}(i)
\]
The first condition is that $\mathbf{v}(j')\in f_{\mathbf{A}(i,j')}^{-1}(\mathbf{w}(i))$
and the second condition is that
$\mathbf{v}(j)\in [- \infty,q_j^i]$
because multiplication by a fixed element is a monotonic function. Then the solution space can be written as 
\[
\bigcup_{j'\in U_i}{\left( X(\mathbf{A}(i,j'),\mathbf{w}(i)) \times  \prod_{j\in \{1,\ldots,m\},j\neq j'}{[- \infty,q_j^i]}\right)} \]
\[= \bigcup_{j'\in U_{i,1}}{\left( (p_{j'}^i,q_{j'}^i]\times  \prod_{j\in \{1,\ldots,m\},j\neq j'}{[- \infty,q_j^i]}\right)} \]
\[
\quad \cup \bigcup_{j'\in U_{i,2}}{\left( [p_{j'}^i,q_{j'}^i]\times  \prod_{j\in \{1,\ldots,m\},j\neq j'}{[- \infty,q_j^i]}\right)}
\]

Taking
$\mathbf{p}_{i,j'}(j)$ and 
$\mathbf{q}_i(j)$ as defined before, we get
\[
[p_{j'}^i,q_{j'}^i] \times \prod_{j\in \{1,\ldots,m\},j\neq j'}{[-\infty,q_j^i]} = [\mathbf{p}_{i,j'},\mathbf{q}_i]
\]
and
\[
(p_{j'}^i,q_{j'}^i] \times \prod_{j\in \{1,\ldots,m\},j \neq j'}{[-\infty,q_j^i]} = \prescript{}{j'}{\lsqparen} \mathbf{p}_{i,j'},\mathbf{q}_i]
\]

Let
\[I_{i,j} = \begin{cases} \{ j\} & \text{if $j\in U_{i,1}$} \\ \emptyset & \text{otherwise} \end{cases}
\]

Using this, we can change the order of intersection with union to get
\begin{align*}
X(\mathbf{A},\mathbf{w})  &= \bigcap_{i=1}^n{\left( \bigcup_{j'\in U_{i,1}}{\prescript{}{j'}{\lsqparen} \mathbf{p}_{i,j'},\mathbf{q}_i]} \cup \bigcup_{j' \in U_{i,2}}{[\mathbf{p}_{i,j'},\mathbf{q}_i]}\right)} \\
 &= \bigcup_{\substack{j'_{i,k} \in U_{i,k} \\ \text{for $1\leq i\leq n$} \\ \text{and $k\in \{1,2\}$}}}{ \bigcap_{1\leq i\leq n, k\in \{1,2\}}{\prescript{}{I_{i,j'_{i,k}}}{\lsqparen} \mathbf{p}_{i,j'_{i,k}}, \mathbf{q}_i]}} \\
\end{align*}

Finally, since the intersection of quasi-intervals can be represented as a quasi-interval as well (Lemma~\ref{intersection of quasi-intervals}), we arrive at
\[
X(\mathbf{A},\mathbf{w})=\bigcup_{\substack{j'_{i,k} \in U_{i,k} \\ \text{for $1\leq i\leq n$} \\ \text{and $k\in \{1,2\}$}}}{ \prescript{}{J_\mathbf{j}}{\Bigglsqparen} \bigvee_{1\leq i\leq n, k\in \{1,2\}}{\mathbf{p}_{i,j'_{i,k}}}, \bigwedge_{1\leq i\leq n}{\mathbf{q}_i}\Biggr]}
\]
where $\mathbf{j} = (j'_{i,k})_{1\leq i\leq n, k\in \{1,2\}}$ and $\ell \in J_\mathbf{j} \subset \{1,\ldots, n\}$ if and only if 
\[
\max_{1\leq i\leq n}{\mathbf{p}_{i,j'_{i,2}}(\ell)} \leq \max_{1\leq i \leq n}{\mathbf{p}_{i,j'_{i,1}}(\ell)}
\]
\end{proof}

\section{Further Research}

The notion of a quasi-interval allows the structure of the solution space of a linear system over a max-blank algebra to be written as a \emph{finite} union of quasi-intervals. This naturally leads to the question of how this notion can be used to express other solution spaces in similarly nice ways. 

While the maximum solution of a linear system is known in many cases, particularly for Heyting algebras (a join-blank algebra in which $\otimes$ is the meet) and max-blank algebras, the entire structure is not known for arbitrary Heyting algebras. 

Also worth investigating is how crucial each of the properties a join-blank algebra satisfies are to Theorem~\ref{join-blank structure theorem}.

\bibliographystyle{ieeetr}

\end{document}